\documentclass[11pt]{article}%
\usepackage[letterpaper, portrait, margin=1in, headheight=0pt]{geometry}
\usepackage{mathtools}
\usepackage{enumerate}
\usepackage{amsmath,amsthm,amssymb,amsfonts}
\usepackage{enumitem}
\usepackage{mathrsfs}
\usepackage{tikz-cd}
\usepackage{multicol}
\usepackage[title]{appendix}
\usetikzlibrary{quotes,angles}

\newcommand{\F}{\mathbb{F}}

\newcommand{\inv}{^{-1}}

\newcommand{\cH}{\mathcal{H}}

\DeclareMathOperator{\spa}{span}

\newtheorem{thm}{Theorem}[subsection]
\newtheorem{lem}[thm]{Lemma}

\theoremstyle{definition}

\theoremstyle{remark}

\title{Associative Algebras with Small Derived Ideal}
\author{Erik Mainellis}
\date{}

\begin{document}

\maketitle

\begin{abstract}
    The paper concerns extra special associative algebras, an analogue of the Heisenberg Lie algebra. In particular, we say that an associative algebra is extra special if its center is equal to its derived ideal and the center is 1-dimensional. In this paper, we classify extra special associative algebras by proving that their structure is equivalent to that of extra special Leibniz algebras. We then characterize their (Schur) multipliers via dimension and completely determine their capability. We connect this with the related notion of unicentral algebras and discuss the problem of classifying extra special diassociative algebras.
\end{abstract}

\section{Introduction}
Let $\F$ be an algebraically closed field of characteristic different from 2. Throughout, we assume that all algebras are finite-dimensional over such a base field.

Recall that a Lie algebra $L$ is called \textit{extra special (or Heisenberg)} if $Z(L) = L'$ and $\dim Z(L) = 1$, where $Z(L)$ denotes the center of $L$ and $L'$ denotes the derived ideal. Extra special Lie algebras have been fully classified; they are exclusively of odd dimension and are characterized by a canonical basis of the form $\{x_1,\dots,x_n,y_1,\dots,y_n,z\}$ with only nonzero products $[x_i,y_i] = -[y_i,x_i] = z$ for $1\leq i\leq n$. The \textit{Schur multipliers} of these algebras have also been studied. In \cite{batten}, it is shown that the dimension of the multiplier $M(L)$ is $2n^2-n-1$ for any extra special Lie algebra $L$ of dimension $2n+1$.

These notions have been generalized to the context of Leibniz algebras. A Leibniz algebra $L$ is \textit{extra special} if $Z(L)=L'$ and $\dim Z(L) =1$, where the center and derived ideal are generalized in the usual manner from Lie algebras. Extra special Leibniz algebras have been fully classified via congruence classes of bilinear forms (see \cite{edal}). In particular, any extra special Leibniz algebra is a central sum of the following five classes of extra special Leibniz algebras:
\begin{enumerate}
    \item[i.] $J_1$ with basis $\{x,z\}$ and nonzero product $xx=z$;
    
    \item[ii.] $J_n$ for $n=2,3,\dots$, with basis $\{x_1,\dots,x_n,z\}$ and nonzero products \begin{align*}
        x_1x_2 = z, ~~~ x_2x_3 = z, ~~~ \cdots ~~~ x_{n-1}x_n = z;
    \end{align*}
    
    \item[iii.] $\Gamma_n$ for $n=2,3,\dots$, with basis $\{x_1,\dots,x_n,z\}$ and nonzero products \begin{align*}
        & x_nx_1 = z, ~~~ x_{n-1}x_2 = -z, ~~~ \cdots ~~~ x_ix_{n-i+1} = (-1)^{n-i+2}z, ~~~\cdots ~~~x_2x_{n-1} = (-1)^nz, \\ & x_nx_2 = z, ~~~ x_{n-1}x_3 = -z, ~~~ \cdots ~~~ x_ix_{n-i+2} = (-1)^{n-i+2}z, ~~~\cdots ~~~x_2x_{n} = (-1)^nz, \\ & x_1x_n = (-1)^{n+1}z;
    \end{align*}

    \item[iv.] $H_2(\lambda)$ with basis $\{x_1,x_2,z\}$ and nonzero products $x_1x_2 = z$, $x_2x_1 = \lambda z$ for $0\neq \lambda \neq 1$;

    \item[v.] $H_{2n}(\lambda)$ for $n=2,3,\dots$, with basis $\{x_1,\dots,x_{2n},z\}$ and nonzero products \begin{align*}
        & x_1x_{n+1} = z, ~~~ x_2x_{n+2}=z, ~~~ \cdots ~~~ x_nx_{2n} = z, \\ & x_{n+1}x_1 = \lambda z, ~~~ x_{n+2}x_2=\lambda z, ~~~ \cdots ~~~ x_{2n}x_n = \lambda z, \\ & x_{n+1}x_2 = z, ~~~ x_{n+2}x_3 = z, ~~~ \cdots ~~~ x_{2n-1}x_n = z,
    \end{align*} where $0\neq \lambda \neq (-1)^{n+1}$.
\end{enumerate} Here, $\lambda\in \F$ is determined up to replacement by $\lambda\inv$. Schur multipliers of Leibniz algebras have been studied in \cite{edal}, \cite{mainellis batten}, and \cite{rogers}. In \cite{edal}, the dimension of $M(L)$ is shown to be $(\dim(L)-1)^2-1$, for any extra special Leibniz algebra $L$, with the following exceptions: $\dim M(J_1) = 1$, $\dim M(J_2) = 4$, and $\dim M(H_2(-1)) = 5$. The same work determines the capability of extra special Leibniz algebras. A Leibniz algebra $L$ is said to be \textit{capable} if $L\cong K/Z(K)$ for some Leibniz algebra $K$. In particular, an extra special Leibniz algebra is capable if and only if it is isomorphic to one of $J_1$, $J_2$, or $H_2(-1)$; it is unicentral otherwise (see \cite{mainellis batten} for more on unicentral Leibniz algebras).

There has been great success in characterizing nilpotent Lie and Leibniz algebras by invariants related to the dimension of their multipliers (see \cite{arab, batten mult, edal mult new, edal mult, hardy, hardy stitz, shamsaki}). The author of the present paper plans to employ a similar methodology for characterizing associative algebras; the work herein is necessary for this effort. We first show that the structure of an extra special Leibniz algebra is the same as that of an extra special associative algebra and, therefore, that their classifications are the same. We then compute the dimension of the Schur multiplier for any extra special associative algebra and thereby obtain a characterization of these multipliers. We determine exactly when extra special associative algebras are capable and when they are unicentral. The paper concludes with a discussion of extra special diassociative algebras, which generalize associative algebras.

\section{Preliminaries}
Let $A$ be an associative algebra. We say that a subspace $S$ of $A$ is a \textit{subalgebra} of $A$ if $ab\in S$ for any $a,b\in S$. A subalgebra $I$ of $A$ is called an \textit{ideal} of $A$ if $ia\in I$ and $ai\in I$ for all $a\in A$ and $i\in I$. We denote by $A'=AA$ the \textit{derived ideal} of $A$, the ideal generated by all products in $A$, and denote by $Z(A)$ the \textit{center} of $A$, the ideal consisting of all $z\in A$ such that $za = az = 0$ for all $a\in A$. We say that an associative algebra $A$ is \textit{extra special} if $Z(A) = A'$ and $\dim(Z(A)) = 1$. An associative algebra $A$ is a \textit{central sum} of subalgebras $A_1,\dots,A_n$ if \[A = \sum_{i=1}^nA_i\] and there is an ideal $Z$ in $A$ such that $A_i\cap A_j = Z$ and $A_iA_j = 0$ for all $i\neq j$. An \textit{extension} of an associative algebra $A$ is a short exact sequence \[0\xrightarrow{} \ker \omega\xrightarrow{} K\xrightarrow{\omega} A\xrightarrow{} 0\] of homomorphisms such that $K$ is an associative algebra. We say that such an extension is \textit{central} if $\ker \omega\subseteq Z(K)$ and \textit{stem} if $\ker \omega\subseteq Z(K)\cap K'$. A \textit{section} of the extension is a linear map $\mu:A\xrightarrow{} K$ such that $\omega\circ\mu = \text{id}_A$.

A pair of associative algebras $(K,M)$ is called a \textit{defining pair} for $A$ if $A\cong K/M$ and $M\subseteq Z(K)\cap K'$. In other words, a defining pair describes a stem extension \[0\xrightarrow{} M\xrightarrow{} K\xrightarrow{\omega} A\xrightarrow{} 0\] where $M = \ker \omega$. Such a pair is a \textit{maximal defining pair} if the dimension of $K$ is maximal. In this case, $K$ is called a \textit{cover} of $A$ and $M$ is called the \textit{multiplier} of $A$, denoted by $M(A)$. In \cite{mainellis batten di}, it is shown that $M(A)$ is isomorphic to $\cH^2(A,\F)$, the second cohomology group with coefficients in the base field $\F$. It is also shown that covers are unique. These results are special cases of ones concerning diassociative algebras.

Finally, let $Z^*(A)$ denote the intersection of all images $\omega(Z(K))$ such that \[0\xrightarrow{} \ker \omega \xrightarrow{} K\xrightarrow{\omega} A\xrightarrow{} 0\] is a central extension of $A$. Clearly $Z^*(A)\subseteq Z(A)$. Following the definition in \cite{mainellis batten di}, however, we say that $A$ is \textit{unicentral} if $Z(A) = Z^*(A)$. A related notion is that of capability; an associative algebra $A$ is \textit{capable} if $A\cong K/Z(K)$ for some associative algebra $K$.

\section{Results}
\subsection{Classification}
\begin{lem}
    Any algebra is an extra special Leibniz algebra if and only if it is an extra special associative algebra under the same multiplication structure.
\end{lem}

\begin{proof}
    Assume that $A$ is an extra special Leibniz algebra with basis $\{x_1,\dots,x_n,z\}$, multiplications $x_ix_j = \alpha_{ij}z$, $\alpha_{ij}\in\F$, and $Z(A) = A' = \spa\{z\}$. We first note that any 3-product in $A$ is zero, and so the Leibniz multiplication structure on $A$ is trivially associative. Moreover, since $A$ is an extra special Leibniz algebra, at least one product involving each $x_i$ is nonzero. Therefore $A$ is an extra special associative algebra under $x_ix_j = \alpha_{ij}z$. The reverse direction follows similarly.
\end{proof}

By virtue of this lemma, the characterization of extra special associative algebras is the same as that of the Leibniz case. The following theorem mimics the classification of the latter, as obtained in \cite{edal}.

\begin{thm}\label{classification}
Any extra special associative algebra is a (unique) central sum of associative algebras with the forms $J_1$, $J_n$, $\Gamma_n$, $H_2(\lambda)$, and $H_{2n}(\lambda)$, for $n=2,3,\dots$, where $\lambda$ is determined up to replacement by $\lambda\inv$.
\end{thm}

\subsection{Multipliers}
We now characterize the multipliers of extra special associative algebras. While these algebras have the same structure as that of the Leibniz case, their multipliers are notably different.

\begin{thm}
    Let $A$ be an extra special associative algebra. Then $\dim M(A) = (\dim(A) - 1)^2-1$ with the exception of $A=J_1$. In particular, $\dim M(J_1) = 1$.
\end{thm}

\begin{proof}
    Let $K$ be the cover of $A$ corresponding to stem extension \[0\xrightarrow{} \ker \omega \xrightarrow{}K\xrightarrow{\omega} A\xrightarrow{}0\] with linear section $\mu:A\xrightarrow{} K$ of $\omega$. Then $K$ is isomorphic to $\mu(A)\oplus \ker \omega$ as a vector space. For the sake of notation, we will identify $x\in A$ with $\mu(x)\in K$. By our classification of extra special algebras, we can choose a canonical basis $\{x_1,\dots,x_n,z\}$ for $A$ such that $Z(A)=A' = \spa\{z\}$. The maximum possible dimension of $K$ could only be attained with a basis of the form \[\{x_i, z, a_{ij}, b_i,c_i\}_{i,j=1,\dots,n}\] and multiplications $x_ix_j = \mu(x_ix_j) + a_{ij}$, $x_iz = b_i$, and $zx_i = c_i$, where $a_{ij},b_i,c_i\in \ker \omega$. We note that $zz=0$ in $K$. Furthermore, a change of basis allows for $a_{11}=0$. By Theorem \ref{classification}, we know that $A$ is either $J_1$, $J_n$, $\Gamma_n$, $H_2(\lambda)$, $H_{2n}(\lambda)$, or a central sum of these classes.

    Suppose $A$ is a central sum of two or more algebras. Then, for any $x_i$, where $1\leq i\leq n$, we can always choose $x_j$ and $x_k$ such that $x_jx_k = \alpha z$ and $x_ix_j=0$ as products in $A$. Here, $1\leq j,k\leq n$ and $0\neq \alpha\in \F$. We compute \begin{align*}
        b_i &= x_iz = x_i\left(\alpha\inv  \mu(x_jx_k)\right) = x_i\left(\alpha\inv  (x_jx_k - a_{jk})\right) \\ &= \alpha\inv  x_i(x_jx_k) = \alpha\inv  (x_ix_j)x_k = \alpha\inv a_{ij}x_k = 0
    \end{align*} for $1\leq i\leq n$. Similarly, for any $x_i$, $1\leq i\leq n$, we can choose $x_j$ and $x_k$ such that $x_jx_k=\alpha z$ and $x_kx_i = 0$ for $1\leq j,k\leq n$ and $0\neq \alpha\in \F$. One thereby computes $c_i = 0$. But this means that all 3-products in $K$ are trivial. Thus, there are no further relations to consider and we obtain $\dim K = (\dim A - 1)^2 + \dim A - 1$, which implies that $\dim M(A) = (\dim(A)-1)^2-1$. If $A$ is not a central sum of two or more algebras, then there are five cases.
    \begin{enumerate}
        \item[i.] Suppose $A = J_1$. Then $K$ is generated by $\{x_1,z,b_1,c_1\}$ and has multiplications $x_1x_1 = z$, $xz = b_1$, and $zx = c_1$. However, one computes $b_1 = x(xx) = (xx)x = c_1$, and so $M(J_1)$ is one-dimensional.
        \item[ii.] Suppose $A=J_n$ for $n\geq 3$. For $2\leq i\leq n$, one computes \begin{align*}
            b_i = x_i(x_{i-1}x_i) = (x_ix_{i-1})x_i = a_{i(i-1)}x_i = 0
        \end{align*} and \begin{align*}
            c_i = (x_{i-1}x_i)x_i = x_{i-1}(x_ix_i) = x_{i-1}a_{ii} = 0.
        \end{align*} We also compute \begin{align*}
            b_1 = x_1(x_2x_3) = (x_1x_2)x_3 = c_3 = 0
        \end{align*} and \begin{align*}
            c_1 = (x_1x_2)x_1 = x_1(x_2x_1) = x_1a_{21} = 0
        \end{align*} and thus $\dim M(J_n) = n^2 - 1 = (\dim(J_n)-1)^2 - 1$ for $n\geq 3$. Now suppose $A=J_2$. One computes $b_2 = c_2 = 0$ via the same logic of the $n\geq 3$ case. We compute \begin{align*}
            b_1 =x_1(x_1x_2) = (x_1x_1)x_2 = a_{11}x_2 = 0
        \end{align*} and \begin{align*}
            c_1 = (x_1x_2)x_1 = x_1(x_2x_1) = x_1a_{21} = 0
        \end{align*} which yields $\{a_{12},a_{21},a_{22}\}$ as a basis for the multiplier. Thus $\dim M(J_2) = 3 = 2^2 - 1$.

        \item[iii.] Suppose $A=\Gamma_n$ for $n\geq 2$.

        If $n$ is odd, then $x_ix_i = a_{ii}$ in $K$ for all $i\neq \frac{n+1}{2}$. For $i\neq \frac{n+1}{2}$, we compute \begin{align*}
            b_i = x_i(x_ix_{n-i+1}) = (x_ix_i)x_{n-i+1} = 0
        \end{align*} for $1\leq i\leq n$ and \begin{align*}
            c_{i} = (x_{n-i+2}x_i)x_{i} = x_{n-i+2}(x_ix_{i}) = 0
        \end{align*} for $2\leq i\leq n$. Here, we have reindexed $c_{n-i+2}$ by $c_i$ for $2\leq i\leq n$. Next, one computes $c_1 = (x_nx_1)x_1 = x_n(x_1x_1) = 0$. Finally, one has \begin{align*}
            b_{\frac{n+1}{2}} = x_{\frac{n+1}{2}}\left(x_{\frac{n-1}{2}}x_{\frac{n+3}{2}}\right) = \left(x_{\frac{n+1}{2}}x_{\frac{n-1}{2}}\right)x_{\frac{n+3}{2}} = (-1)^{\frac{n+3}{2}}zx_{\frac{n+3}{2}} = (-1)^{\frac{n+3}{2}}c_{\frac{n+3}{2}} = 0
        \end{align*} and \begin{align*}
            c_{\frac{n+1}{2}} = \left(x_{\frac{n-1}{2}}x_{\frac{n+3}{2}}\right)x_{\frac{n+1}{2}} = x_{\frac{n-1}{2}}\left(x_{\frac{n+3}{2}}x_{\frac{n+1}{2}}\right) = (-1)^{\frac{n+1}{2}}x_{\frac{n-1}{2}}z = (-1)^{\frac{n+1}{2}}b_{\frac{n-1}{2}} = 0.
        \end{align*}

        If $n$ is even, then $x_ix_i = a_{ii}$ in $K$ for all $i\neq \frac{n+2}{2}$. For $i\neq \frac{n+2}{2}$, we compute \begin{align*}
            b_i = x_i(x_ix_{n-i+1}) = (x_ix_i)x_{n-i+1} = 0
        \end{align*} for $1\leq i\leq n$ and \begin{align*}
            c_i = (x_{n-i+2}x_i)x_i = x_{n-i+2}(x_ix_i) = 0
        \end{align*} for $2\leq i\leq n$. Furthermore, one has $c_1 = (x_nx_1)x_1 = x_n(x_1x_1) = 0$, \begin{align*}
            b_{\frac{n+2}{2}} = x_{\frac{n+2}{2}}\left(x_{\frac{n}{2}}x_{\frac{n+2}{2}}\right) = \left(x_{\frac{n+2}{2}}x_{\frac{n}{2}}\right)x_{\frac{n+2}{2}} = (-1)^{\frac{n+2}{2}}zx_{\frac{n+2}{2}} = (-1)^{\frac{n+2}{2}}c_{\frac{n+2}{2}},
        \end{align*} and \begin{align*}
            c_{\frac{n+2}{2}} = \left(x_{\frac{n}{2}}x_{\frac{n+2}{2}}\right)x_{\frac{n+2}{2}} = x_{\frac{n}{2}}\left(x_{\frac{n+2}{2}}x_{\frac{n+2}{2}}\right) = (-1)^{\frac{n+2}{2}}x_{\frac{n}{2}}z = (-1)^{\frac{n+2}{2}}b_{\frac{n}{2}} = 0,
        \end{align*} which implies that $b_{\frac{n+2}{2}}=0$.
        
        Therefore, $\dim M(\Gamma_n) = n^2-1 = (\dim(\Gamma_n)-1)^2-1$ for $n\geq 2$.

        \item[iv.] Suppose $A=H_2(\lambda)$, where $0\neq \lambda\neq 1$. Then $c_1 = (x_1x_2)x_1 = x_1(x_2x_1) = \lambda b_1$ and also $c_1 = \lambda\inv(x_2x_1)x_1= \lambda\inv x_2(x_1x_1) = 0$. Hence $b_1 = c_1 = 0$. Next, one has $c_2 = (x_1x_2)x_2 = x_1(x_2x_2) = 0$ and $c_2 = \lambda\inv(x_2x_1)x_2 = \lambda\inv x_2(x_1x_2) = \lambda\inv b_2$, which implies that $b_2 = c_2 = 0$. Therefore, $M(H_2(\lambda))$ has basis $\{a_{12},a_{21},a_{22}\}$, and so the dimension of this multiplier is $(\dim(H_2(\lambda)) - 1)^2 - 1$.

        \item[v.] Suppose $A=H_{2n}(\lambda)$ for $n\geq 2$. Then, for $1\leq i\leq n$, one has \begin{align*}
            b_i = \lambda\inv x_i(x_{n+i}x_i) = \lambda\inv(x_ix_{n+i})x_i = \lambda\inv zx_i = \lambda\inv c_i
        \end{align*} and \begin{align*}
            c_i = \lambda\inv(x_{n+i}x_i)x_i = \lambda\inv x_{n+i}(x_ix_i) = 0
        \end{align*} which implies that $b_i = c_i = 0$ for $1\leq i\leq n$. For $n+1\leq i\leq 2n-1$, we compute \begin{align*}
            b_i = x_i(x_ix_{i-n+1}) = (x_ix_i)x_{i-n+1} = 0
        \end{align*} and \begin{align*}
            c_i = (x_ix_{i-n+1})x_i = x_i(x_{i-n+1}x_i) = x_ia_{(i-n+1)i} = 0.
        \end{align*} Finally, one has $b_{2n} = \lambda\inv x_{2n}(x_{2n}x_n) = \lambda\inv(x_{2n}x_{2n})x_n = 0$ and $c_{2n} = (x_nx_{2n})x_{2n} = x_n(x_{2n}x_{2n}) = 0$, which implies that $\dim M(H_{2n}(\lambda)) = (\dim(H_{2n}(\lambda)) - 1)^2 - 1$.
    \end{enumerate}
\end{proof}

\subsection{Capability}
Let $A$ be an associative algebra and $0\xrightarrow{} R\xrightarrow{} F\xrightarrow{} A\xrightarrow{} 0$ be a free presentation of $A$. In \cite{mainellis batten di}, it is shown that $\overline{\pi}(Z(\overline{F}))\subseteq \omega(Z(K))$ for every central extension \[0\xrightarrow{} \ker \omega \xrightarrow{} K\xrightarrow{\omega} A\xrightarrow{} 0,\] where $\overline{F} = \frac{F}{FR+RF}$, $\overline{R} = \frac{R}{FR+RF}$, and \[0\xrightarrow{} \overline{R}\xrightarrow{} \overline{F}\xrightarrow{\overline{\pi}} A\xrightarrow{} 0\] is the central extension induced by our free presentation. The same paper also shows that $Z^*(A) = \overline{\pi}(Z(\overline{F})) = \omega(Z(K))$ when our initial central extension is stem. With this setup, we are able to prove the following criterion for the capability of $A$.

\begin{thm}
    An associative algebra $A$ is capable if and only if $Z^*(A) = 0$.
\end{thm}

\begin{proof}
    If $A$ is capable, there exists a central extension \[0\xrightarrow{} Z(K)\xrightarrow{} K\xrightarrow{\omega} A\xrightarrow{} 0\] for some associative algebra $K$. By the aforementioned results, this means that $\overline{\pi}(Z(\overline{F}))\subseteq \omega(Z(K))$. But $\omega(Z(K)) = 0$ since $Z(K) = \ker \omega$. Thus $\overline{\pi}(Z(\overline{F})) = 0$, which implies that $Z^*(A)=0$. Conversely, assume that $Z^*(A) = 0$. Then $\overline{\pi}(Z(\overline{F}))=0$ implies that $Z(\overline{F})\subseteq \overline{R}$, which means that $Z(\overline{F})=\overline{R}$. We thus have a central extension \[0\xrightarrow{} Z(\overline{F})\xrightarrow{} \overline{F}\xrightarrow{\overline{\pi}} A\xrightarrow{} 0\] and so $A$ is capable.
\end{proof}

We are now ready to fully determine the capability of extra special associative algebras.

\begin{thm}
    Let $A$ be an extra special associative algebra. Then $A$ is capable if and only if $A\cong J_1$. Otherwise, $A$ is unicentral.
\end{thm}

\begin{proof}
    Let $K$ be the cover of $A$ corresponding to stem extension \[0\xrightarrow{} M(A)\xrightarrow{} K\xrightarrow{\omega}A\xrightarrow{} 0\] with linear section $\mu:A\xrightarrow{} K$. We first note that $Z(K) = M(A)$ if and only if $A=J_1$. Otherwise, $Z(K) = M(A) + \spa\{\mu(Z(A))\}$. In the case of $A=J_1$, we have $A=K/Z(K)$. In other words, $A$ is capable. When $A\neq J_1$, we have $Z^*(A) = \omega(Z(K)) = Z(A) \neq 0$ since $A$ is extra special. Thus, $A$ is unicentral but not capable.
\end{proof}

\section{Generalization}
Loday's \textit{associative dialgebras} \cite{loday}, or \textit{diassociative algebras}, are characterized by two bilinear multiplications, usually denoted by $\dashv$ and $\vdash$, that generalize associativity via the following five relations:
\begin{align*}
    (x\vdash y)\vdash z = x\vdash(y\vdash z) && (x\dashv y)\dashv z = x\dashv(y\dashv z) \\
    (x\dashv y)\vdash z = x\vdash (y\vdash z) & &(x\dashv y)\dashv z = x\dashv (y\vdash z) & \\
    (x\vdash y)\dashv z = x\vdash (y\dashv z)
\end{align*}
These algebras generalize the classic relation $[x,y] = xy - yx$ between associative and Lie algebras by defining a Leibniz algebra with multiplication $xy = x\dashv y - y\vdash x$. Furthermore, any associative algebra can be thought of as a diassociative algebra in which $x\dashv y = x\vdash y$. An \textit{extra special diassociative algebra} can be defined in an analogous way to that of the previous algebraic contexts; one notes that the structures of the center and derived algebra are significantly more complicated, but we should still require that they be 1-dimensional.

The problem of classifying extra special diassociative algebras is an open one. This is because the problem of obtaining canonical forms for the congruence classes of pairs of matrices also appears to be an open one. Explicitly, a pair of matrices $(M_1,M_2)$ is \textit{congruent} to another pair $(N_1,N_2)$ if there exists an invertible matrix $P$ such that $P^TM_iP = N_i$ for $i=1,2$. The classification of extra special diassociative algebras could be obtained via the methodology of the Leibniz and associative classifications (again, see \cite{edal}) if these congruence classes were known.

\end{document}